\documentclass[amssymb, 11pt]{amsart}
\usepackage{amsmath}
\usepackage{amsthm}
\usepackage{amsfonts}
\newcommand{\ee}{\mathbb{E}}
\newcommand{\pp}{\mathbb{P}}
\newtheorem{thm}{Theorem}[section]
\newtheorem{theorem}[thm]{Theorem}

\newtheorem{lemma}[thm]{Lemma}

\usepackage{varioref}
\labelformat{section}{Section~#1}
\labelformat{figure}{Figure~#1}
\labelformat{table}{Table~#1}

\usepackage[numbers]{natbib}
\date{}

\def\ba#1{\begin{align*}#1\end{align*}}
\def\ban#1{\begin{align}#1\end{align}}
\newcommand{\IE}{\ee}
\newcommand{\IP}{\pp}
\def\norm#1{\Vert#1\Vert}
\def\abs#1{|#1|}
\newcommand{\eq}{\eqref}

\newcommand{\hh}{\tilde{h}}

\begin{document}

\title [Exponential approximation and Stein's method] {Exponential approximation and Stein's method of exchangeable pairs}

\author{Jason Fulman}
\address{Department of Mathematics\\
        University of Southern California\\
        Los Angeles, CA, 90089, USA}
\email{fulman@usc.edu}

\author{Nathan Ross}
\address{Department of Statistics\\
        University of California, Berkeley\\
        Berkeley, CA, 94720, USA}
\email{ross@stat.berkeley.edu}


\date{Version of July 20, 2012}

\begin{abstract}
We derive a new result for exponential approximation using Stein's method of exchangeable pairs.
As an application, an
exponential limit theorem with error term is derived for $|Tr(U)|^2$, where $Tr(U)$ denotes the trace
of a matrix chosen from the Haar measure of the unitary group $U(n,\mathbb{C})$.
\end{abstract}

\maketitle

\section{Introduction} \label{intro}

The purpose of this paper is to further develop exponential approximation, using Stein's method of exchangeable pairs. The first use
of exchangeable pairs in exponential approximation was in the paper \cite{CFR}, which studied the spectrum of the Bernoulli-Laplace Markov chain. Unfortunately the results in \cite{CFR}, which use the Kolmogorov metric, are very complicated and seem quite hard to apply in other examples.
We provide approximation results similar to those in~\cite{CFR}, but which are significantly easier to compute. In particular, we do not see how to apply the results of \cite{CFR} to the example in this paper.

We work in a ``smooth" test function metric, but also provide bounds in the Kolmogorov metric, which is defined for random
variables $W$ and $Z$ to be
\ba{
d_K(W,Z)=\sup_{t\in\mathbb{R}}|\IP(W\leq t)-\IP(Z\leq t)|.
}
Our main theoretical result is the following theorem (see also Theorem~\ref{thmba} below).
\begin{theorem}\label{main}
Let $Z$ be a mean one exponential random variable.
If $W\geq0$  is a random variable with finite second moment and $(W,W')$ is
an exchangeable pair such that for some $a>0$ and sigma-field $\mathcal{F}\supseteq\sigma(W)$,
\ban{
\IE[W'-W|\mathcal{F}]=-a(W-1)+R, \label{lin}
}
then for all $\delta>0$,
\ba{
&d_K(W,Z)\leq \frac{8}{\delta}\IE\left|W-\frac{\IE[(W'-W)^2|\mathcal{F}]}{2a}\right|+\frac{2}{\delta}\abs{\IE W -1}  \\
&\qquad+\left(\frac{(5-6/e)}{\delta}+\frac{3}{\delta^2}\right)\frac{\IE\abs{W'-W}^3}{a}+\frac{8}{\delta}\frac{\IE\abs{R}}{a} +\delta/2.
}
\end{theorem}
\noindent{\it Remark:} The theorem is stated with a choice of $\delta$ in order to simplify the error bound,
but it is obvious that in applications $\delta$ should be chosen to minimize the bound.

One of the attractive points about this result is that the terms are very similar to those which one encounters in normal approximation. To see the parallels, here is a normal approximation theorem of Rinott and Rotar \cite{RR2} (in the Kolmogorov metric).

\begin{theorem} \label{rinrot} Let $(W,W')$ be an
 exchangeable pair of real random variables such that $\ee(W)=0, \ee(W^2)=1$ and
$\ee(W'|W) = (1-a)W + R(W) $ with $0< a <1$. Then for all real $x_0$,
\begin{eqnarray*} & & \left| \pp(W \leq x_0) - \frac{1}{\sqrt{2 \pi}}
\int_{-\infty}^{x_0} e^{-\frac{x^2}{2}} dx \right|\\ & \leq &
\frac{6}{a} \sqrt{Var(\ee[(W'-W)^2|W])} + 19 \frac{\sqrt{\ee(R^2)}}{a} +
6 \sqrt{\frac{1}{a} \ee|W'-W|^3}. \end{eqnarray*} \end{theorem}

For normal approximation, there are algebraically natural examples related to Markov chain spectra and random matrices (see \cite{F1}, \cite{F2}), which are perfectly suited for the computation of terms such as $Var(\ee[(W'-W)^2|W])$ and $|W'-W|^3$. This is why we believe the bound in Theorem \ref{main} will be useful for exponential approximation.

Indeed,
in \ref{rmt}, we consider the random variable $W=|Tr(U)|^2$, where $Tr$ denotes trace and $U$ is from the Haar measure of the unitary
group $U(n,\mathbb{C})$. Since $Tr(U)$ converges to a complex normal \cite{DS}, it follows that $|Tr(U)|^2$ converges to an exponential with mean
1. In studying the correspondence between unitary eigenvalues and zeros of the Riemann zeta function, it is conjectured in \cite{CD} that the
convergence of $|Tr(U)|^2$ to an exponential limit is extremely rapid, more precisely that there are positive $c, \delta$ such that for all $n \geq 1, t
\geq 0$, \[ | \pp(|Tr(U)|^2 \geq t) - e^{-t} | \leq cn^{-\delta n}.\] The authors suggest that this should follow from methods in Johansson's remarkable
paper \cite{J}. This seems challenging to make rigorous, particularly if one wants to make $c,\delta$ explicit. In \ref{rmt}, we give the first rigorous
and explicit error term for this problem,
proving that the Kolmogorov distance between $|Tr(U)|^2$ and a standard mean 1 exponential is at most $2^{9/4}/\sqrt{n}$.
Another approach to this result might be to use the multivariate central limit theorems in \cite{DoSt}; we do not pursue this here.

To close the introduction, we mention some related results using Stein's method for exponential approximation.
Aside from \cite{CFR} which we already mentioned, a recent paper of Chatterjee and Shao \cite{CS}
(and similar results in Section~13.4 of the text~\cite{CGS}),
use Stein's method of exchangeable pairs for exponential approximation.
However our approach is quite different than theirs since
they assume the exchangeable pair $(W,W')$ satisfies \[ \ee(W'-W|W) = 1/c_0 + R(W), \]
with $c_0$ a positive constant, rather than the
linearity condition~\eqref{lin} assumed here.
Another approach to Stein's method for exponential approximation is the ``equilibrium distribution" coupling
for which we refer the reader to the papers by Pek\"{o}z and R\"{o}llin \cite{PR1}, \cite{PR2}, and to the references therein.
For the generator method (in the more general context of chi-squared approximation), one can consult \cite{Lu} or \cite{Re}. We note that the examples in \cite{Lu} and \cite{Re} are about independent random variables, whereas the example in the current paper involves dependence.
Finally, the introductory survey \cite{Ros} has some discussion of these approaches in the wider context of Stein's method.

\section{General theorem} \label{general}

The purpose of this section is to prove Theorem \ref{main} from the introduction.
We first prove an intermediate result which can be thought of as an approximation result for ``smooth" test functions.
\begin{theorem}\label{thmba}
Let $Z$ be a mean one exponential random variable.
If $W\geq0$  is a random variable with finite second moment and $(W,W')$ is
an exchangeable pair such that for some $a>0$ and sigma-field $\mathcal{F}\supseteq\sigma(W)$,
\ban{ \label{waad}
\IE[W'-W|\mathcal{F}]=-a(W-1)+R,
}
then for all twice differentiable functions $h$ with $\norm{h'},\norm{h''}<\infty$,
\ban{
&\left|\IE h(W)-\IE h(Z)\right|\leq 4\norm{h'}\IE\left|W-\frac{\IE[(W'-W)^2|\mathcal{F}]}{2a}\right|+\norm{h'}\abs{\IE W-1} \label{wabd1} \\ 
&\qquad+\left(2(5-6/e) \norm{h'}+3\norm{h''}\right)\frac{\IE\abs{W'-W}^3}{4a}+4\norm{h'}\frac{\IE\abs{R}}{a}.\label{wabd2}  
}
\end{theorem}

The proof of Theorem~\ref{thmba} roughly follows
the usual development of Stein's method of exchangeable pairs for distributional approximation.
Specifically, for $W$ the random variable of interest and $Z$ having the exponential distribution, we
want to bound
$\left|\IE h(W)-\IE h(Z)\right|$
for functions $h$ in some predetermined family of test functions (here,
twice differentiable functions $h$ with $\norm{h'},\norm{h''}<\infty$).
Typically, this program has three components.
\begin{enumerate}
\item[1.] Define a \emph{characterizing operator} $\mathcal{A}$ for the exponential distribution which
has the property that
\ba{
\IE \mathcal{A}f(Z)=0
}
for all $f$ in a large enough class of functions if and only if $Z\sim Exp(1)$.

\item[2.] For functions $h$ in the class of interest, define $f_h$ to solve
\ban{\label{cheq}
\mathcal{A}f_h(x)=h(x)-\IE h(Z).
}
\item[3.] Using \eq{cheq}, note that
\ba{
\abs{\IE h(W)-\IE h(Z)} = \abs{\IE \mathcal{A}f_h(W)}.
}
Now use properties of the solutions $f_h$ and the auxiliary randomization of
exchangeable pairs to bound this last term.
\end{enumerate}

The next lemma takes care of Items~1 and~2 and also provides
the bounds on the
solutions $f_h$ needed for Item~3 in the program above.
The proof of Theorem~\ref{thmba} is immediately after the proof of the lemma.

\begin{lemma}\label{lemstbd}
Let $Z$ be a mean one exponential random variable.
If $h$ is a function such that the following integrals are well defined,
then
\ban{
f(w)&=f_h(w)=-\frac{e^w}{w}\int_w^\infty (h(x)-\IE h(Z))e^{-x}dx \label{stnsol}
}
solves the differential equation
\ban{
wf'(w)-(w-1)f(w)=h(w)-\IE h(Z). \label{stneq}
}
If $h$ is absolutely continuous with $\norm{h'}<\infty$, then
\ban{
\norm{f}\leq \left(1+\frac{2}{e}\right)\norm{h'}, \hspace{5mm} \norm{f'}\leq  2\norm{h'}. \label{abcobd}
}
If in addition, $h'(0)=0$ and $h'$ is absolutely continuous with $\norm{h''}<\infty$,
then
\ba{
\norm{f''}\leq (5-6/e) \norm{h'}+3\norm{h''}.
}
\end{lemma}

\begin{proof}
The fact that \eqref{stnsol} solves \eqref{stneq} is straightforward to verify.
Now, using that $\IE h(Z)=\int_0^\infty h(x)e^{-x} dx$,
we can rewrite \eqref{stnsol} as
\ban{
f(w)&= -\frac{e^w}{w}\int_w^\infty h(x) e^{-x}dx+\frac{1}{w}\int_0^\infty e^{-x} h(x)dx \notag \\
&=-\frac{e^w(1-e^{-w})}{w}\int_w^{\infty}h(x)e^{-x}dx+\frac{1}{w}\int_0^w h(x)e^{-x}dx. \label{rewrite}
}
To prove \eqref{abcobd}, first note that since translating $h$ by a constant leaves $f$ unchanged, we may (and do) assume
without loss of generality that $h(0)=0$, so that
$h(x)\leq \norm{h'}\abs{x}$. Using this fact and also that
$\int xe^{-x} dx = -e^{-x}(x+1)$ in the equality below, we find
\ba{
\abs{f(w)}&\leq\norm{h'}\left[\frac{e^w(1-e^{-w})}{w}\int_w^{\infty}xe^{-x}dx+\frac{1}{w}\int_0^w xe^{-x}dx\right] \\
	&=\norm{h'}\left[\frac{(1-e^{-w})(w+2)}{w}-e^{-w}\right].
}
To bound this last expression, we compare derivatives to find
\ba{
(1-e^{-w})(w+2)-we^{-w}\leq(1+2/e) w, \hspace{5mm} w\geq0,
}
which yields the first assertion of \eqref{abcobd}.

For the second assertion note that \eq{stneq} implies that
\ban{
f'(w)=\frac{h(w)}{w}-\left(\frac{(1-w)f(w)+\IE h(Z)}{w}\right). \label{stnsol1}
}
Since $|h(x)| \leq \norm{h'}\abs{x}$, the first term
of \eq{stnsol1} is bounded in absolute value by $\norm{h'}$, so we only need
to find an appropriate bound on the term of \eq{stnsol1} that is in parentheses.
We have by \eq{rewrite} that
\begin{equation}\label{exint}
\begin{split}
\frac{(1-w)f(w)+\IE h(Z)}{w}&=  \left(\frac{e^w-1}{w}-\frac{e^{w}-1-w}{w^2}\right) \int_w^\infty h(x)e^{-x}dx  \\
&\qquad+\frac{1}{w^2}\int_0^wh(x)e^{-x}dx.
\end{split}
\end{equation} One easily checks that \[ \frac{e^w-1}{w}-\frac{e^{w}-1-w}{w^2} \geq 0.\]
Now taking the absolute value of \eq{exint}, using the triangle inequality,
bounding $\abs{h(x)}\leq\norm{h'}\abs{x}$, and using $\int xe^{-x} dx = -e^{-x}(x+1)$ , we find
\ba{
&\left|\frac{(1-w)f(w)+\IE h(Z)}{w}\right| \\
&\quad\leq \norm{h'}\left(\frac{(w+1)(w+e^{-w}-1)}{w^2}+\frac{1-(w+1)e^{-w}}{w^2}\right)=\norm{h'},
}
which now yields the second assertion of \eqref{abcobd}.

To prove the final statement of the lemma, take
the derivative of \eqref{stnsol1} using the expression \eqref{exint} to find
\ban{
f''(w)&=\frac{h'(w)}{w}+\frac{(w-2)h(w)}{w^2} \label{2der1}\\
&+\frac{2-(w^2-2w+2)e^w}{w^3}\int_w^\infty h(x)e^{-x} dx
+\frac{2}{w^3}\int_0^w h(x)e^{-x}dx. \label{2der2}
}
To bound these expressions we first note that since $h'(0)=h(0)=0$,
\ba{
\abs{h(x)}\leq\min\{\norm{h'} \abs{x}, \norm{h''}x^2/2\}, \text{ and } \abs{h'(x)}\leq \norm{h''}\abs{x}
}
and in particular, $\abs{h(x)}$ is bounded above by both terms appearing in the minimum.
Thus, the absolute value of the right hand side of \eqref{2der1} is bounded above by
\ba{
\norm{h''} + \min\{\abs{w/2-1}\norm{h''}, \abs{1-2/w}\norm{h'}\}&\leq \norm{h''}+\max\{\norm{h'},\norm{h''}\}, \\
	&\leq 2\norm{h''}+\norm{h'}
}
where we have used that $\min\{\abs{w/2-1}, \abs{1-2/w}\}\leq 1$.

We bound the second term \eqref{2der2} differently according to $w\geq1$ or $w<1$.
Suppose that $w \geq 1$. Then note that $(w^2-2w+2) e^w \geq e^w \geq 2$.
Using that $\abs{h(x)}\leq\norm{h'}\abs{x}$ and $\int xe^{-x} dx = -e^{-x}(x+1)$, we find
the absolute value of \eqref{2der2}
is bounded above by
\ban{
&\norm{h'}\left[\frac{((w^2-2w+2)e^w-2)(w+1)e^{-w}}{w^3}+\frac{2(1-(w+1)e^{-w})}{w^3}\right] \notag \\
&\qquad=\norm{h'}\frac{w^3-w^2+4-4(w+1)e^{-w})}{w^3} \notag \\
&\qquad\leq \norm{h'} \frac{w^3+3(1-(w+1)e^{-w})}{w^3}, \label{t1t}
}
where we have used that $1-w^2\leq0$.
By comparing derivatives we find
\ba{
1-(w+1)e^{-w}\leq(1-2/e)w^3, \hspace{5mm} w\geq1,
}
so that \eqref{t1t} (and hence~\eqref{2der2}) is bounded above by $(4-6/e)\norm{h'}$ for $w\geq1$.

If $0\leq w<1$, then \[ (w^2-2w+2) e^w \geq (w^2-2w+2) \left( 1+w+\frac{w^2}{2} \right) = \frac{w^4+4}{2} \geq 2.\] Using that $\abs{h(x)}\leq\norm{h''}x^2/2$
and  $\int x^2 e^{-x}dx=-e^{-x}(2+2x+x^2)$, we find
the absolute value of \eqref{2der2}
is bounded above by
\ban{
&\norm{h''}\left[\frac{((w^2-2w+2)e^w-2)(2+2w+w^2)e^{-w}}{2w^3}+\frac{2-(2+2w+w^2)e^{-w}}{w^3}\right] \notag \\
&\qquad=\norm{h''}\frac{w^4+8-4e^{-w}(2+2w+w^2)}{2w^3}.\label{t2t}
}
Again by comparing derivatives we find
\ba{
w^4+8-4e^{-w}(2+2w+w^2)\leq 2w^3, \hspace{5mm} 0\leq w<1,
}
so that \eqref{t2t} (and hence~\eqref{2der2}) is bounded above by $\norm{h''}$ for $0<w\leq 1$.
\end{proof}

\begin{proof}[Proof of Theorem \ref{thmba}]
We show that for $h$ as in theorem, $\abs{\IE h(W)-\IE h(Z)}$ is appropriately bounded.
We would like to follow the program outlined at the beginning of the section,
but in order to apply the bounds of Lemma~\ref{lemstbd}, we must have $h'(0)=0$, which is not assumed in
Theorem~\ref{thmba}. We circumvent this problem by replacing $h$ with $\hh(x)=h(x)-xh'(0)$, and we have
\ban{
\abs{\IE h(W)-\IE h(Z)}&\leq\abs{\IE \hh(W)-\IE\hh(Z)}+\abs{h'(0)}\abs{\IE W- \IE Z} \notag \\
	&\leq \abs{\IE \hh(W)-\IE\hh(Z)}+\norm{h'}\abs{\IE W- 1}. \label{htil}
}
In order to bound $\abs{\IE \hh(W)-\IE\hh(Z)}$, we use Lemma~\ref{lemstbd}
in conjunction with Item~3 of the program outlined at the beginning of this section to
show that the absolute value of
\ban{
\IE [Wf'(W)-(W-1)f(W)]  \label{bdth1}
}
is appropriately bounded, where $f$ satisfies \eq{stneq} with $h$ replaced by $\hh$.

Using exchangeability and the linearity condition \eq{waad}, we observe that
\ba{
\IE [(W'-W)(f(W)-f(W'))] &=2\IE [f(W)(W'-W)] \\
&= -2a\IE [(W-1) f(W)] + 2\IE [R f(W)],
}
so that \eqref{bdth1} is equal to
\ba{
\IE [Wf'(W)] - (2a)^{-1}\IE [(W'-W)(f(W')-f(W))] -a^{-1}\IE [R f(W)].
}
We rewrite this expression as
\ba{
&\IE \left[ f'(W)\left(W-\frac{\IE[(W'-W)^2|\mathcal{F}]}{2a}\right) \right] \\
&\qquad -\IE \left[ \frac{(W'-W)}{2a}\int_0^{W'-W} [f'(W+t)-f'(W)] dt \right] -a^{-1}\IE [R f(W)].
}
Now taking the absolute value of this last expression,
we find that \eq{bdth1} in absolute value is bounded above by
\ba{
&\norm{f'}\IE\left|W-\frac{\IE[(W'-W)^2|\mathcal{F}]}{2a}\right| \\
&\qquad+\norm{f''}\IE \left[ \frac{\abs{W'-W}}{2a}\int_0^{W'-W}\abs{t} dt \right]
+\norm{f}\frac{\IE\abs{R}}{a}. 
}
The result now easily follows after applying the bounds of Lemma~\ref{lemstbd}, with $h$ replaced by $\hh$,
noting that $\norm{\hh''}=\norm{h''}$ and $\norm{\hh'}=\norm{h'-h'(0)}\leq2\norm{h'}$, and recalling~\eqref{htil}.
\end{proof}

We are now in a position to prove Theorem~\ref{main}. First define the function
for $t,x\geq0$, and $\delta>0$,
\begin{equation}\label{htd}
 h_{t,\delta}(x) = \left\{
     \begin{array}{ll}
       1, &  x \leq t-\delta\\
       1-\frac{2(x-t+\delta)^2}{\delta^2}, &t-\delta<x\leq t-\delta/2 \\
       \frac{2(x-t)^2}{\delta^2},&t-\delta/2<x\leq t \\
       0, & x > t
     \end{array}
   \right.
\end{equation}

The next lemma states some important facts regarding the use of $h_{t,\delta}$ in our framework.
\begin{lemma}\label{lemhtd}
If $t\geq0$, $\delta>0$, and  $h_{t,\delta}$ is defined by~\eqref{htd}, then
\ba{
\norm{h_{t,\delta}}=1, \hspace{5mm} \norm{h_{t,\delta}'}=2/\delta, \hspace{5mm} \norm{h_{t,\delta}''}=4/\delta^2.
}
If $W\geq0$ is a random variable and $Z$ has the exponential distribution with mean one, then
\ban{
d_K(W,Z)\leq\sup_{t\geq0}\abs{\IE h_{t,\delta}(W)-\IE h_{t,\delta}(Z)}+\delta/2. \label{smo}
}
\end{lemma}

\begin{proof}
The first assertion follows from direct computation. For the second, note that
\ba{
\IP(W\leq t)-\IP(Z\leq t)&\leq \IE h_{t+\delta,\delta}(W)-\IP(Z\leq t) \\
	&=\IE h_{t+\delta,\delta}(W)-\IE h_{t+\delta,\delta}(Z) + \IE h_{t+\delta,\delta}(Z)-\IP(Z\leq t) \\
	&\leq \abs{\IE h_{t+\delta,\delta}(W)-\IE h_{t+\delta,\delta}(Z)}+\int_t^{t+\delta}h_{t+\delta,\delta}(x)e^{-x}dx.
}
Since $e^{-x}\leq 1$ for $x>0$, we find by direct computation that
\ba{
\int_t^{t+\delta}h_{t+\delta,\delta}(x)e^{-x}dx\leq\int_t^{t+\delta}h_{t+\delta,\delta}(x)dx= \delta/2.
}
Taking supremums, we have shown
\ban{
\sup_{t\geq0}[\IP(W\leq t)-\IP(Z\leq t)]\leq\sup_{t\geq0}\abs{\IE h_{t,\delta}(W)-\IE h_{t,\delta}(Z)}+\delta/2.\label{oha}
}
A similar argument
starting from
\ba{
\IP(Z\leq t)-\IP(W\leq t)&\leq \IP(Z\leq t)-\IE h_{t,\delta}(W)
}
establishes \eqref{oha} with the left hand side replaced by
\[ \sup_{t\geq0}[\IP(Z\leq t)-\IP(W\leq t)] , \] which proves the lemma.
\end{proof}

\begin{proof}[Proof of Theorem~\ref{main}]
First apply Theorem~\ref{thmba} with $h$ replaced by $h_{t,\delta}$
to obtain a bound on $\sup_{t\geq0}\abs{\IE h_{t,\delta}(W)-\IE h_{t,\delta}(Z)}$.
From this point, the result follows from the bounds of Lemma~\ref{lemhtd}
and~\eqref{smo}.
\end{proof}

\section{Exponential approximation of $|Tr(U)|^2$} \label{rmt}

The main purpose of this section is to prove the following result.

\begin{theorem} \label{main2} Let $W=|Tr(U)|^2$, where $U$ is from the Haar measure of $U(n,\mathbb{C})$. Then for $n \geq 8$,
the Kolmogorov distance between $W$ and an exponential with mean one is at most $2^{9/4}/\sqrt{n}$.
\end{theorem}

To construct an exchangeable pair to be used in our application, we use the heat kernel of $U(n,\mathbb{C})$. This has proved useful in other Stein's method problems about random matrices \cite{F3}, \cite{FR}, \cite{DoSt}.  See \cite{G}, \cite{Ro} for a detailed discussion of heat kernels on compact Lie groups. The papers \cite{L},\cite{Liu}, \cite{R} illustrate combinatorial uses of heat kernels on compact Lie groups,
and \cite{Liu} also discusses the use of the heat kernel for finite groups.

The heat kernel on a compact Lie group $G$ is defined by setting for $x,y \in G$ and $t \geq 0$, \begin{equation} \label{heatk} K(t,x,y) =
\sum_{n \geq 0} e^{-\lambda_n t} \phi_n(x) \overline{\phi_n(y)}, \end{equation} where the $\lambda_n$ are the eigenvalues
of the Laplacian repeated according to multiplicity, and the $\phi_n$ are an orthonormal basis of eigenfunctions of $L^2(G)$; these can be
taken to be the irreducible characters of $G$.

We use the following properties of the heat kernel, where $\Delta$ denotes the Laplacian of $G$. Part 1 of Lemma \ref{spectral}
is from page 198 of \cite{G}. Part 2 of Lemma \ref{spectral} is immediate from the expansion \eqref{heatk}. Part 3 of
Lemma \ref{spectral} is Lemma 2.5 of \cite{DoSt}.

\begin{lemma} \label{spectral} Let $G$ be a compact Lie group, $x,y \in G$, and $t \geq 0$.
\begin{enumerate}
\item $K(t,x,y)$ converges and is non-negative for all $x,y,t$.
\item $\int_{y \in G} K(t,x,y) dy = 1$, where the integration is with respect to Haar measure of $G$.
\item For smooth $\phi$, as $t \rightarrow 0$, one has that
\[ \int_{y \in G} K(t,x,y) \phi(y) dy = \phi(x) + t(\Delta \phi)(x) + O(t^2). \]
\end{enumerate}
\end{lemma}

The symmetry in $x$ and $y$ of $K(t,x,y)$ shows that the heat kernel is a reversible Markov
process with respect to the Haar measure of $G$. It is a standard fact \cite{RR}, \cite{Stn} that
reversible Markov processes lead to exchangeable pairs $(W,W')$. Namely suppose one has a Markov chain with
transition probabilities $K(x,y)$ on a state space $X$, and that the Markov chain is reversible with respect to a
probability distribution $\pi$ on $X$. Then given a function $f$ on $X$, if one lets $W=f(x)$ where $x$ is chosen
from $\pi$ and $W'=f(x')$ where $x'$ is obtained by moving from $x$ according to $K(x,y)$, then $(W,W')$
is an exchangeable pair. In the special case of the heat kernel on a compact Lie group $G$, given a function $f$ on $G$,
one can construct an exchangeable pair $(W,W')$ by letting $W=f(U)$ where $U$ is chosen from Haar measure, and $W'=f(U')$,
where $U'$ is obtained by moving time $t$ from $U$ via the heat kernel. To define the exchangeable pair
$(W,W')$ used in this paper, we further specialize by setting $f(U)=|Tr(U)|^2$.

If $\lambda$ is an integer partition, and $m_j$ denotes the multiplicity of part $j$ in $\lambda$, we define
$p_{\lambda}(U) = \prod_j Tr(U^j)^{m_j}$. For example, $p_{5,3,3}(U) = Tr(U^5) Tr(U^3)^2$. Typically we
suppress the $U$ and use the notation $p_{\lambda}$.

The next three lemmas are from Rains \cite{R}; here $\nabla f \cdot \nabla g$ is defined by
\[ \nabla f \cdot \nabla g = \frac{1}{2}[ \Delta(fg) - g \Delta f -f \Delta g ] .\]

\begin{lemma} \label{one} $\Delta_{U(n)} p_1 = -n p_1$. \end{lemma}

\begin{lemma} \label{nab} For all integers $k$ and $l$ (not necessarily positive), and unitary $U$,
\[ (\nabla p_k(U)) \cdot (\nabla p_l(U)) = -kl \cdot p_{k+l}(U) .\]
\end{lemma}

\begin{lemma} \label{product} For all unitary matrices $U$ and class functions $f_1,\cdots,f_k$
\begin{eqnarray*}
& & \Delta \left( \prod_{1 \leq i \leq k} f_i(U) \right) \\
& = & \left( \prod_{1 \leq i \leq k} f_i(U) \right) \left(\sum_{1 \leq i \leq k} \frac{\Delta f_i(U)}{f_i(U)} + 2 \sum_{1 \leq i < j \leq k} \frac{(\nabla f_i(U)) \cdot (\nabla f_j(U))}{f_i(U) f_j(U)} \right).
\end{eqnarray*}
\end{lemma}

The final lemma is a moment computation from \cite{DS}.

\begin{lemma} \label{mom} Let $U$ be Haar distributed on $U(n,\mathbb{C})$. Let $(a_1,\cdots,a_k)$ and
$(b_1,\cdots,b_k)$ be vectors of non-negative integers. Then one has that for all $n \geq \sum_{i=1}^k (a_i+b_i),$
\[ \ee \left[ \prod_{j=1}^k Tr(U^j)^{a_j} \overline{Tr(U^j)^{b_j}} \right] = \delta_{\vec{a} \vec{b}} \prod_{j=1}^k j^{a_j} a_j!. \]
\end{lemma}

Throughout we let $W(U) = |Tr(U)|^2 = p_1(U) \overline{p_1(U)}$.

Lemma \ref{Ucond1} computes the conditional expectation $\ee[W'-W|U]$.

\begin{lemma} \label{Ucond1} \[ \ee[W'-W|U] = 2nt(1-W) + O(t^2). \]
\end{lemma}

\begin{proof} Applying part 3 of Lemma \ref{spectral},
\begin{eqnarray*} \ee[W'|U] & = &  W + t \Delta (p_1 \overline{p_1}) + O(t^2) \\
& = & W + t [ \overline{p_1} \Delta (p_1) + p_1 \Delta (\overline{p_1}) + 2 (\nabla p_1) \cdot (\nabla \overline{p_1}) ]
+ O(t^2) \\
& = & W + t [ -2n p_1 \overline{p_1} + 2n ] + O(t^2).
\end{eqnarray*} The second equality was Lemma \ref{product}, and the final equality used Lemmas \ref{one} and \ref{nab}.
\end{proof}

Lemma \ref{Ucond2} computes $\ee[(W'-W)^2|U]$.

\begin{lemma} \label{Ucond2}
\[ \ee[(W'-W)^2|U] = t [ -2 p_2 \overline{p_{1,1}} - 2 \overline{p_2} p_{1,1} + 4nW] + O(t^2) .\]
\end{lemma}

\begin{proof} Clearly
\[ \ee[(W'-W)^2|U] = \ee[(W')^2|U] - 2W \ee[W'|U] + W^2.\]

By part 3 of Lemma \ref{spectral},
\[ \ee[(W')^2|U]  =  W^2 + t \Delta [p_{1,1} \overline{p_{1,1}}] + O(t^2) .\]
Using Lemma \ref{product}, and then Lemmas \ref{one} and \ref{nab}, one computes that
\begin{eqnarray*}
& & \Delta [p_{1,1} \overline{p_{1,1}}] \\
& = & p_{1,1} \overline{p_{1,1}} \left[ \frac{2 \Delta p_1}{p_1} + \frac{2 \Delta \overline{p_1}}{\overline{p_1}}
+ \frac{2 \nabla p_1 \cdot \nabla p_1}{p_{1,1}} + \frac{2 \nabla \overline{p_1} \cdot \nabla \overline{p_1}}{\overline{p_{1,1}}}
+ \frac{8 \nabla p_1 \cdot \nabla{\overline{p_1}}}{p_1 \overline{p_1}} \right] \\
& = & -4n p_{1,1} \overline{p_{1,1}} - 2p_2 \overline{p_{1,1}} - 2 \overline{p_2} p_{1,1} + 8n p_1 \overline{p_1}.
\end{eqnarray*} Thus \[ \ee[(W')^2|U] = W^2 + t \left[ -4n p_{1,1} \overline{p_{1,1}} - 2p_2 \overline{p_{1,1}} - 2 \overline{p_2} p_{1,1} + 8n p_1 \overline{p_1} \right] + O(t^2) .\]

By Lemma \ref{Ucond1},
\[ -2 W \ee[W'|U] = -2W^2 + t[-4nW + 4nW^2] + O(t^2).\]
Thus \[ \ee[(W')^2|U] - 2W \ee[W'|U] + W^2 = t[ -2 p_2 \overline{p_{1,1}} - 2 \overline{p_2} p_{1,1} + 4nW] + O(t^2).\]
\end{proof}

Next we compute expected values of low order moments of $W'-W$.

\begin{lemma} \label{Ulow} Suppose that $n \geq 8$. Then
\begin{enumerate}
\item $\ee(W'-W)^2 = 4nt + O(t^2)$.
\item $\ee(W'-W)^4 = O(t^2)$.
\item $\ee|W'-W|^3 = O(t^{3/2})$.
\end{enumerate}
\end{lemma}

\begin{proof} Lemma \ref{Ucond2} implies that \[ \ee(W'-W)^2 = t \ee \left[ -2 p_2 \overline{p_{1,1}} - 2 \overline{p_2} p_{1,1} + 4nW \right] + O(t^2).\] By Lemma \ref{mom}, $\ee[p_2 \overline{p_{1,1}}] = 0$, $\ee[\overline{p_2} p_{1,1}] = 0$, and $\ee[W]=1$; the first part of the lemma follows.

For part 2, first note that since
\[ \ee[(W'-W)^4] = \ee(W^4) - 4 \ee(W^3W') + 6 \ee[W^2(W')^2] - 4 \ee[W (W')^3] + \ee[(W')^4],\]
exchangeability of $(W,W')$ gives that
\begin{eqnarray*}
\ee(W'-W)^4 & = & 2 \ee(W^4) -8 \ee(W^3W') + 6 \ee[W^2(W')^2] \\
& = & 2 \ee(W^4) -8 \ee[W^3 \ee[W'|U]] + 6 \ee[W^2 \ee[(W')^2|U]].
\end{eqnarray*}

By Lemma \ref{Ucond1},
\begin{eqnarray*}
-8 \ee[W^3 \ee[W'|U]] & = & -8 \ee[W^3 (W + t(2n-2nW) + O(t^2))] \\
& = & - 8 \ee(W^4) + t \ee[ -16nW^3 + 16nW^4 ] + O(t^2) \\
& = & - 8 \ee(W^4) +tn [ -16(6) + 16(24) ] + O(t^2) \\
& = & - 8 \ee(W^4) + 288 tn + O(t^2),
\end{eqnarray*} where the penultimate equality used Lemma \ref{mom}.

By the proof of Lemma \ref{Ucond2}, and then Lemma \ref{mom},
\begin{eqnarray*} & & 6 \ee[W^2 \ee[(W')^2|U]] \\
& = & 6 \ee \left[ W^2 \left( W^2+t[-4nW^2 - 2p_2 \overline{p_{1,1}} - 2 \overline{p_2} p_{1,1} +8nW] +O(t^2) \right) \right] \\
& = & 6 \ee[W^4] +tn \ee[ -24 W^4 + 48 W^3] + O(t^2) \\
& = & 6 \ee[W^4] + tn [ -24(24) + 48(6)] + O(t^2) \\
& = & 6 \ee[W^4] - 288 tn + O(t^2).
\end{eqnarray*}

Summarizing, it follows that
\begin{eqnarray*}
\ee(W'-W)^4 & = & 2 \ee(W^4) -8 \ee[W^3 \ee[W'|U]] + 6 \ee[W^2 \ee[(W')^2|U]] \\
& = & O(t^2), \end{eqnarray*} proving part 2 of the lemma.

For part 3 of the lemma, one uses the Cauchy-Schwarz inequality to obtain that
\[ \ee|W'-W|^3 \leq \sqrt{\ee(W'-W)^2 \ee(W'-W)^4}.\] Part 3 then follows from parts 1 and 2 of the lemma.
\end{proof}

Now we proceed to the proof of Theorem \ref{main2}.

\begin{proof}[Proof of Theorem \ref{main2}] By Lemma \ref{Ucond1}, one can apply Theorem \ref{main} with $a=2nt$. By Lemma \ref{Ucond2}, and the
triangle inequality,
\begin{eqnarray*}
& & \ee \left|W - \frac{\ee[(W'-W)^2|W]}{2a} \right| \\
& = & \ee \left|W - \frac{t[-2p_2 \overline{p_{1,1}} - 2 \overline{p_2}p_{1,1} + 4nW]}{4nt} + \frac{O(t^2)}{4nt} \right| \\
& = & \frac{1}{2n} \ee |p_2 \overline{p_{1,1}} + \overline{p_2} p_{1,1}| + O(t) \\
& \leq & \frac{1}{2n} \sqrt{\ee(p_{2,2} \overline{p_{1,1,1,1}} + 2p_{2,1,1} \overline{p_{2,1,1}} + \overline{p_{2,2}} p_{1,1,1,1})}+O(t).
\end{eqnarray*} By Lemma \ref{mom}, this is $\frac{\sqrt{2}}{n} + O(t)$; letting $t \rightarrow 0$ gives an upper bound
\[ \ee \left|W - \frac{\ee[(W'-W)^2|W]}{2a} \right| \leq \frac{\sqrt{2}}{n}.\]

The second term in Theorem \ref{main} is 0 since by Lemma \ref{mom}, $\ee(W)=1$.

To bound $\frac{\ee|W'-W|^3}{a}$, note by Lemma \ref{Ulow} that $\ee|W'-W|^3 = O(t^{3/2})$. Since $a=2nt$, the term
$\frac{\ee|W'-W|^3}{a}$ tends to 0 as $t \rightarrow 0$.

Finally, note from Lemma \ref{Ucond1} that $R=O(t^2)$. Since $a=4nt$, it follows that
\[\frac{ \ee|R|}{a} \leq \frac{ \sqrt{\ee(R(W)^2)}}{a} = O(t) \] tends to 0 as $t \rightarrow 0$.

Summarizing, by letting $t \rightarrow 0$, Theorem~\ref{main} implies that for $\delta>0$,
\ba{
d_K(W,Z)\leq\frac{8 \sqrt{2}}{\delta n}+\delta/2.
}
Choosing $\delta=4\cdot2^{1/4}/\sqrt{n}$, yields the claimed result.
\end{proof}

{\it Remarks:}
\begin{enumerate}
\item The moments of the random variable $|Tr(U)|^2$ have a combinatorial interpretation. Indeed, from \cite{R2}
one has for all positive integers $l,n$ that
\[ \pp(L_n \leq l) = \frac{1}{n!} \int_{U(l,\mathbb{C})} |Tr(U)|^{2n} .\] Here $L_n$ is the length of the
longest subsequence of a random permutation on $n$ symbols.

\item The technique used in this section can be used to prove that for positive integers $k$,
$|Tr(U^k)|^2/k$ tends to an exponential with mean 1, for $U$ a Haar distributed unitary matrix from
$U(n,\mathbb{C})$, as $n \rightarrow \infty$. The bookkeeping is quite tedious, so we do not carry this out.
\end{enumerate}

\section*{Acknowledgements} Fulman was partially supported by a Simons Foundation Fellowship. We thank Eric Rains for
helpful correspondence.

\end{document}